\documentclass{amsart}

\usepackage{hyperref,   amssymb,amsbsy,amsfonts,amsthm,latexsym,amsopn,amstext, amsxtra,euscript,amscd}

\newtheorem{thm}{Theorem}
\newtheorem{lem}{Lemma}
\newtheorem{cor}{Corollary}

\newtheorem{exmp}{Example}
\newtheorem{rem}{Remark}

\newcommand{\set}[1]{{\left\{#1\right\}}}    

\def\C{\mathbb C}
\def\Q{\mathbb Q}
\def\F{\mathbb F}

\def\Z{\mathbb Z}
\def\O{\mathcal O}
\def\R{\mathcal R}

\def\La{\Lambda}
\def\Th{\theta}
\def\th{\theta}

\def\om{\omega}

\def\<{<}
\def\>{>}
\def\O{\mathcal O}

\def\L{\Lambda}

\title[Theta functions for codes over imaginary quadratic fields]{Theta functions and symmetric weight enumerators for codes over imaginary quadratic fields}

\author{T. Shaska}
\address{Department of Mathematics and Statistics, Oakland University, Rochester, MI, 48309.}

\author{C. Shor}
\address{Department of Mathematics, Western New England University, Springfield, MA 01119}

\begin{document}

\begin{abstract}
In this paper we continue the study of codes over imaginary quadratic fields and their weight enumerators and theta functions.  We present new examples of non-equivalent codes over rings of characteristic $p=2$ and $p=5$ which have the same theta functions.  We also look at a generalization of codes over imaginary quadratic fields, providing examples of non-equivalent pairs with the same theta function for $p=3$ and $p=5$.
\end{abstract}

\maketitle

\section{Introduction}

Let $\ell>0$ be a square-free integer congruent to 3 modulo 4,  $K=\Q(\sqrt{-\ell })$ be the imaginary quadratic field, and $\O_K$ its ring of integers.  Codes, Hermitian lattices, and their theta-functions over rings $\R:=\O_K / p \O_K$, for small primes $p$, have been studied by many authors, see \cite{ba}, \cite{ms1}, \cite{ms2}, among others. In \cite{ba}, explicit descriptions of theta functions and MacWilliams identities are given for $p=2, 3$.

If $p=2$ then the image $\O_K/2\O_K$ of the projection $\rho_{\ell}:\O_K \rightarrow \O_K/2\O_K $ is $\F_4$  (resp., $\F_2 \times \F_2$) if $\ell \equiv 3 \mod 8$ (resp.,  $\ell \equiv 7 \mod 8$).  Let $\R$  be a ring isomorphic to $\F_4$  or $\F_2 \times \F_2$ and $C$ a linear code over $\R$ of length $n$. An admissible level $\ell$ is an $\ell$ such that  $\ell \equiv 3 \mod 8$ if
$\R$ is isomorphic to $\F_4$ or $\ell \equiv 7 \mod 8$ if $\R$ is isomorphic to $\F_2 \times \F_2$. Fix an admissible $\ell$ and define $\L_{\ell} (C):= \set { x \in \O_K^n : \rho_\ell (x) \in C}.$
Then, from \cite{ch}, the \textbf{level $\ell$ theta function} $\Th_{\L_{\ell} (C) } (q) $ of the lattice $\L_{\ell} (C)$ is given as the symmetric weight enumerator $swe_C $ of $C$, evaluated on the theta functions defined on cosets of $\O_K/2\O_K$.  Interestingly, an example of two codes of length $n=3$ defined over $\mathbb{F}_2\times\mathbb{F}_2$ that have different symmetric weight enumerator polynomials and the same theta function was given.

In \cite{SS} and \cite{sh_sh}, the above constructions are extended to consider theta functions of codes defined over $\R=\mathbb{F}_p\times\mathbb{F}_p$ or $\mathbb{F}_{p^2}$, for any prime $p$.  These constructions are extended further in \cite{dkl} for codes over $\R=\mathbb{F}_{p^e}\times\mathbb{F}_{p^e}$ or $\mathbb{F}_{p^{2e}}$.  In particular, the theta function $\th_{\La_{\ell}(C)}(q)$ of such a code is equal to the complete weight enumerator of the code evaluated at theta functions of cosets of $\R$.

These constructions naturally lead to two questions.  
\\

i) How do the theta functions of the same code $C$ vary for different levels $\ell$?

ii) How do the theta functions of different codes vary?  Can different codes have the same theta functions for any (or all) levels $\ell$?
\\

The first question is answered for $p=2$ in \cite{SV} and for general $p$ in \cite{sh_sh}.  For a code $C$ defined over $\R$ and for all admissible $\ell, \ell^\prime$ such that $\ell>\ell^\prime,$ we have that \[ \Th_{\Lambda_\ell(C)}(q) =\Th_{\Lambda_{\ell^\prime}(C)}(q)+ {\mathcal O} \left(q^{\frac {\ell^\prime+1}{4}}\right). \]

The second question has been approached a few ways.  Suppose $p=2$ and let $C$ be a code of length $n$ defined over $R=\mathbb{F}_2\times\mathbb{F}_2$ or $\mathbb{F}_4$.  In \cite{SV}, which is included here as Theorem~\ref{thm:p=2-codes-theorem}, for $\ell$ large enough, non-equivalent codes have different theta functions.  Interestingly, the method used in the proof of that theorem does not work for larger primes, though in \cite{sh_sh} we were able to generalize the above method to begin to describe the situations (in terms of $\ell$ and $n$) where more examples of different codes with the same theta function could exist.

The focus of this paper is to present explicit examples of pairs of different codes with the same theta function.  For the case of $p=2$, a family of pairs of codes with this property defined over $\mathbb{F}_2\times\mathbb{F}_2$ of lengths $n\geq2$ and level $\ell=7$ is given in Theorem~\ref{thm:family-of-pairs-of-codes}.  Another pair of codes defined over $\mathbb{F}_2\times\mathbb{F}_2$ of length $n=3$ and level $\ell=7$ is given in Example~\ref{ex:new-length3-codes}.  And a third pair of codes, this time defined over $\mathbb{F}_5\times\mathbb{F}_5$, of length $n=2$ and level $\ell=19$ is given in Example~\ref{ex:p=5-codes-same-theta}.

Additionally, for the cases of $p=3$ and $p=5$, we also consider structures which we call ``$\mathbb{F}_p$-submodule codes,'' which are $\mathbb{F}_p$-submodules of $\R^n$ rather than $\R$-submodules.  This looser restriction leads to many examples which we present for $p=3$ and $p=5$.  While these aren't actually codes, we include the results here with the idea that there is perhaps a way to use these structures to create pairs of codes the same property.


This paper is organized as follows.  In Section~\ref{sec:preliminaries}, we describe lattices, theta functions, and the method known as Construction A by which one creates a theta function from a code for any prime $p$.  In Section~\ref{sec:non-equiv-codes-same-theta-fn}, we then consider the question of whether there exist different codes with the same theta function, first looking at $p=2$. Specific examples are given.  The cases of $p=3$ and $p=5$ follow, and included within is the notion of a $\mathbb{F}_p$-submodule code.  More specific examples are given.  Finally, Section\ref{sec:concluding-remarks} contains some concluding remarks and ideas for future work.


\bigskip

\section{Preliminaries}\label{sec:preliminaries}

In this section we give a brief overview of codes over imaginary quadratic fields.  Most of the material of this section can be found in \cite{SV} or \cite{sh_sh}. 

Let $\ell>0$ be a square free integer and $K=\Q(\sqrt{-\ell })$ be the imaginary quadratic field with discriminant $d_K$. Recall that $ d_K=   -\ell$ if $\ell\equiv 3 \mod 4,$ and $d_K=  -4\ell$ otherwise.  Let $\O_K$ be the ring of integers of $K$.

\subsection{Theta functions of lattices over $K$}

A  lattice $\La$ over $K$ is an $\O_K$-submodule of $K^n$ of full rank.  The inner product is defined as   $x \cdot y := \sum_{i=1}^n x_i   y_i$.  
   The theta series of a lattice $\L$ in $K^n$ is defined as 
\[  \Th_\L \left(\tau \right) := \sum_{z \in \L} e^{\pi i\tau z\cdot\bar{z}},\] 
where $ \tau \in H =\set { z \in \C : Im(z)>0}$ and $\bar y$ denotes component-wise complex conjugation.  Let $q = e^{\pi i \tau}$. Then, $ \Th_\L(q) = \sum_{z \in \L} q^{z\cdot\bar{z}}$.

The one dimensional theta series (or Jacobi's theta series) and its shadow are given by
\[
\begin{split}
\th_3(q)   & =\sum_{n\in\Z} q^{n^2}, \\
 \th_2(q)  & =\sum_{n\in\Z} q^{(n+1/2)^2}=\sum_{n\in\Z+\frac{1}{2}}q^{n^2}.
 \end{split}
 \]
Let $\ell\equiv 3 \mod 4$ and $d$ be a positive number such that $\ell=4d-1$. Then, $-\ell\equiv 1 \mod 4$.
This implies that the ring of integers is $\O_K=\Z[\om_\ell]$, where $\om_\ell=\frac{-1+\sqrt{-\ell}}{2}$ and
$\om_\ell^2 + \om_\ell+d=0$. The principal norm form of $K$ is given by
\begin{equation}\label{eq1}
Q_d(x,y) = |x-y\om_\ell|^2 = x^2+xy+dy^2.
\end{equation}
%
%
%

The structure of $\O_K/p\O_K$ depends on the value of $\ell$. For $p=2$, 
\begin{equation}
 \O_K/2\O_K =
 \begin{cases}
   \F_2\times\F_2 & \text{if $\ell\equiv7$ (mod }$8)$, \\
   \F_4 & \text{if $\ell\equiv3$ (mod }$8)$. \\
 \end{cases}
\end{equation}

For $p$ an odd prime, 
\begin{equation}
 \O_K/p\O_K =
 \begin{cases}
   \F_p\times\F_p & \text{if $(\frac{-\ell}{p})=1$}, \\
   \F_{p^2} & \text{if $(\frac{-\ell}{p})=-1$}, \\
   \F_p+u\F_p \text{ with $u^2=0$} & \text{if $p\mid\ell$}.
 \end{cases}
\end{equation}
where  $\left(\displaystyle\frac{a}{p}\right)$ is the Legendre symbol.  From here forward we assume that  $p\nmid\ell$.

For integers $a$ and $b$ and a prime $p$, let $\La_{a,b}$ denote the coset
$a-b\om_\ell+p\O_K$.  The theta series associated to this coset is
\begin{equation}
\begin{split}
\th_{\La_{a,b}}(q)  &=  \sum_{m,n\in\Z}q^{|a+mp-(b+np)\om_\ell|^2} \\
& =  \sum_{m,n\in\Z}q^{Q_d(mp+a,np+b)}  =  \sum_{m,n\in\Z}q^{p^2Q_d(m+a/p, n+b/p)}.
\end{split}
\end{equation}
For a prime $p$ and an integer $j$, consider the one-dimensional theta series
\begin{equation}\label{eqn:one-dim-series}
\th_{p,j}(q):=\sum_{n\in\Z}q^{(n+j/2p)^2}.
\end{equation}
Note that $\th_{p,j}(q)=\th_{p,k}(q)$ if and only if $j\equiv\pm k \mod 2p$.  The following were proven in \cite{sh_sh}:

\begin{lem}\label{lemma:in-terms-of-one-dim-series}
One can write $\th_{\La_{a,b}}(q)$ in terms of one-dimensional theta series defined above in Eq.~\eqref{eqn:one-dim-series}. In particular,
\begin{equation}\th_{\La_{a,b}}(q)=\th_{p,b}(q^{p^2\ell})\th_{p,2a+b}(q^{p^2})+\th_{p,b+p}(q^{p^2\ell})\th_{p,2a+b+p}(q^{p^2}).
\end{equation}
\end{lem}

\begin{lem}\label{lemma-congruences}
For any integers $a,b,m,n$, if the ordered pair $(m,n)$ is congruent modulo $p$ to one of $(a, b), (-a-b, b),
(-a, -b), (a+b, -b)$, then \[\th_{\La_{m,n}}(q)=\th_{\La_{a,b}}(q).\]
\end{lem}

The Klein 4-group  generated by matrices
\[ \begin{pmatrix}-1&0\\0&-1\end{pmatrix} \text{   and   }\begin{pmatrix}1&1\\0&-1\end{pmatrix}\]
acts on $(Z/pZ)^2$. The orbits form equivalence classes  on $\Z^2$. This equivalence is given by
\[(a,b)\sim(m,n) \text{ if } (m,n) \equiv  (a,b), \, (-a-b, b), \, (-a, -b), \text{ or } (a+b, -b) \mod p.\]
By Lemma \ref{lemma-congruences}, if $(a, b)\sim (m, n)$, then \[\th_{\La_{a,b}}(q)=\th_{\La_{m,n}}(q).\]
Then we have the following result:
\begin{cor}\label{cor-num-theta-fns}
For   any odd prime $p$, the set $\set{\th_{\La_{a,b}}(q) : a, b \in \Z}$ contains at most
$\frac{(p+1)^2}{4}$ elements.
\end{cor}

The next result determines in what cases we have exactly $\frac{(p+1)^2}{4}$ theta functions. See \cite[Corollary 3.8]{dkl} for details.

\begin{thm}\label{bound_on_d_for_dimension}
For any odd prime $p$ and any  $d>p^2$, the set $\set{\th_{\La_{a,b}}(q) : a,b\in\Z}$ spans a
$\frac{(p+1)^2}{4}$ dimension vector space in $\Z[[q]]$.  \end{thm}

\begin{rem}
i) As a corollary of the Theorem we have that Lemma \ref{lemma-congruences} is an ``if and only if'' statement
for large enough $d$. 

ii) The bound for $d$ given in Theorem \ref{bound_on_d_for_dimension} is not sharp.  For instance, using an implementation in the Sage computer algebra system, one finds that for $d=2$, there are $\frac{(p+1)^2}{4}$ equivalence classes for all primes $p\leq19$. 
\end{rem}


\subsection{Theta functions of codes over $\O_K / p\O_K$}\label{sec:theta_functions_of_codes}
Let $p\nmid\ell$ and \[\R := \O_K / p\O_K =\left\{a+b\om : a, b\in\F_p,\, \om^2+\om+d=0\right\},\] as above.  We have the map
\[\rho_{\ell, p}:\O_K \rightarrow \O_K/ p\O_k =: \R\]
A linear code $C$ of length $n$ over $\R$ is an $\R$-submodule of $\R^n$. The dual is defined as
\[C^\bot=\set{u\in \R^n : u\cdot \bar{v}=0 \text{ for all } v \in C},\]  where $\bar{v}$ denotes (component-wise) complex conjugation.
If $C=C^\bot$ then $C$ is called self-dual.  We define
\[\L_{\ell}(C):= \left\{ u=(u_1,\dots,u_n) \in \O_K^n : (\rho_{\ell, p}(u_1),\dots,\rho_{\ell,p}(u_n)) \in C \right\}.\]
In other words, $\L_{\ell}(C)$ consists of all vectors in $\O_K^n$ in the inverse image of $C$, taken component-wise by $\rho_{\ell,p}$.
This method of lattice construction is known as \emph{Construction A}.

For notation, let $r_{a+pb+1}=a-b\om$, so $\R=\set{r_1,\dots,r_{p^2}}.$  For a codeword $u=(u_1,\dots,u_n)\in\mathcal{R}^n$ and $r_i \in\mathcal{R}$, we define the counting function 
\[ n_i(u) :=\#\set{j : u_j=r_i}.\]
The complete weight enumerator of the $\R$ code $C$ is the polynomial
\begin{equation}
cwe_C(z_1, z_2, \dots, z_{p^2})=\sum_{u\in C}z_1^{n_1(u)}z_2^{n_2(u)}\dots z_{p^2}^{n_{p^2}(u)}.
\end{equation}
We can use this polynomial to find the theta function of the lattice $\La_\ell(C)$. For a proof of the
following result see \cite{SS}.
\begin{lem}\label{lem1} Let $C$ be a code defined over $\R$ and $cwe_C$ its complete weight enumerator as above. Then,
\[\th_{\L_\ell(\mathcal{C})}(q) = cwe_\mathcal{C}(\th_{\La_{0,0}}(q),\th_{\La_{1,0}}(q),\th_{\La_{2,0}}(q),\dots,\th_{\La_{p-1,0}}(q),\th_{\La_{0,1}}(q),\dots,\th_{\La_{p-1,p-1}}(q))\]
\end{lem}

\begin{exmp}
For $p=2$, we have $$\th_{\La_{\ell}(C)}(q)=cwe_C(\th_{\La_{0,0}}(q),\th_{\La_{1,0}}(q),\th_{\La_{0,1}}(q),\th_{\La_{1,1}}(q)).$$  Since $\th_{\La_{0,1}}(q)=\th_{\La_{1,1}}(q)$ (by Lemma \ref{lemma-congruences}), we can define the symmetric weight enumerator $swe_C$ by $$swe_C(X,Y,Z)=cwe_C(X,Y,Z,Z)$$ to get $$\th_{\La_\ell(\mathcal{C})}(q)=swe_\mathcal{C}(\th_{\La_{0,0}}(q),\th_{\La_{1,0}}(q),\th_{\La_{0,1}}(q)).$$  These three theta functions are referred to as $A_d(q), C_d(q),$ and $G_d(q)$ in \cite{ch} and \cite{SV}.\end{exmp}

More generally, for odd $p$, the complete weight enumerator takes $p^2$ arguments corresponding to the $p^2$ lattices $\La_{a,b}(q)$ and their theta functions.  By Theorem \ref{bound_on_d_for_dimension}, for $\ell$ large enough, there are only $\frac{(p+1)^2}{4}$ different theta functions among these $p^2$ lattices.  As above with $p=2$, we define the symmetric weight enumerator of a code in terms of the complete weight enumerator, using the same variable for lattices that have the same theta series.  

\begin{exmp}\label{ex:p=3-swe}
For the case where $p=3$, from \cite[Remark 2.2]{SS}, we have four theta functions corresponding to the lattices
$\La_{a,b}$, namely $\th_{\La_{0,0}}(q)$, $\th_{\La_{1,0}}(q)$, $\th_{\La_{0,1}}(q)$, $\th_{\La_{1,1}}(q).$
We define the symmetric weight enumerator to be $$swe_\mathcal{C}(X,Y,Z,W)=cwe_\mathcal{C}(X,Y,Y,Z,W,Z,Z,Z,W).$$ One then has 
\[
\begin{split}
\th_{\La_\ell(C)} (q) \,  =  \, & cwe_\mathcal{C}(\th_{\La_{0,0}}(q),\th_{\La_{1,0}}(q),\th_{\La_{2,0}}(q),\th_{\La_{0,1}}(q),\dots,\th_{\La_{2,2}}(q)), \\
 \, =\,  & swe_\mathcal{C}(\th_{\La_{0,0}}(q),\th_{\La_{1,0}}(q),\th_{\La_{0,1}}(q),\th_{\La_{1,1}}(q)).
\end{split}
\]
\end{exmp}



\section{Non-equivalent codes with the same theta function}\label{sec:non-equiv-codes-same-theta-fn}
For a fixed prime $p$, let $C$ be a linear code over $\R=\F_{p^2}$ or $\F_p\times\F_p$ of length $n$ and
dimension $k$. An admissible level $\ell$ is an integer $\ell$ such that $\O_K/p\O_K$ is isomorphic to
$\R$.  For an admissible $\ell$, let $\L_{\ell} (C) $ be the corresponding lattice as in the previous section.
Then, the \textbf{level $\ell$ theta function} $\Th_{\L_{\ell} (C) } ( q) $ of the lattice $\L_{\ell} (C)$
is determined by the complete weight enumerator $cwe_C $ of $C$, evaluated on the  theta functions defined on
cosets of $\O_K/p\O_K$.

Two natural questions arise.  First, how do these theta functions vary for different levels $\ell$ and the same code $C$?  And second, how do they vary for different codes $C$ and the same level $\ell$?

In \cite[Theorem 11]{sh_sh}, we give a satisfactory answer to the first question, which we reproduce here.

\begin{thm}\cite[Theorem 11]{sh_sh}.\label{thm1}
Let $C$ be a code defined over $\R$.  For all admissible $\ell, \ell'$ with $\ell<\ell'$ the following holds
\[\th_{\La_\ell(C)}(q)=\th_{\La_{\ell'}(C)}(q)+\O(q^{\frac {\ell+1} 4}).\]
\end{thm}

The rest of this paper is concerned with the second question.  In particular, we are interested in understanding when non-equivalent codes\footnote{For our context here, by ``non-equivalent'' codes, we mean codes having different symmetric weight enumerator polynomials.}  have the same theta function.  That is, we are looking for codes $C_1$ and $C_2$ defined over $\R$, along with some level $\ell$, such that $swe_{C_1}\neq swe_{C_2}$ and $\th_{\L_\ell(C_1)}(q)=\th_{\L_\ell(C_2)}(q)$.

 This is primarily motivated by an example in \cite{ch} of two codes of length $n=3$ defined over $\R=\mathbb{F}_2\times\mathbb{F}_2$ which have the same theta function for $\ell=7$ but different theta functions for larger values of $\ell$.  For convenience, we reproduce the results here.

\begin{exmp}(From \cite[Section 6, Example (4)]{ch}.)\label{ex:chua-length3-codes}For $p=2$, let $\mathcal{R}=\{0,1,\omega,1+\omega\}$.  Consider the codes $C_1$ and $C_2$ of length $n=3$ defined by 
\begin{center}$\begin{array}{rcl}
C_1&=&\omega\langle(0,1,1)\rangle+(\omega+1)\langle(0,1,1)\rangle^\perp, \\ 
C_2&=&\omega\langle(0,0,1)\rangle+(\omega+1)\langle(0,0,1)\rangle^\perp.\end{array}$\end{center}  If $\ell\equiv7 \text{ (mod 8)}$, then $\mathcal{R}=\mathbb{F}_2\times\mathbb{F}_2$, so $\omega^2+\omega=0.$  The codes, along with their weight enumerator polynomials and theta series for various levels $\ell$ are given in Table \ref{tab:chua-example}.  Note that since $p=2$, for any code $C$, $swe_C(X,Y,Z)=cwe_C(X,Y,Z,Z)$.

In particular, these two codes, which have different weight enumerator polynomials, have the same theta series for $\ell=7$ and different theta series for $\ell>7$.\end{exmp}  


In this section, we will first work with $p=2$ and then consider $p=3$ and $p=5$ later.

\subsection{The case of $p=2$}\label{subsec:p=2}
Fix $p=2$ and consider a code $C$ of length $n$ over $\R$, for $\R=\mathbb{F}_2\times\mathbb{F}_2$ or $\mathbb{F}_4$.  The theta function of $C$, $\th_{\L_\ell(C)}(q)$, is equal to its symmetric weight enumerator polynomial, which is a homogeneous polynomial in 3 variables of degree $n$, evaluated at certain theta functions as in Lemma~\ref{lem1}.  To address the question of whether two codes can have the same theta function, we first consider the question of whether two homogeneous degree $n$ polynomials can have the same theta function.

Of course, not every homogeneous degree $n$ polynomial is the weight enumerator polynomial of some code.  However, if the degree $n$ polynomials all give rise to different theta functions for a particular level $\ell$, then we can conclude that the codes of length $n$ give rise to different theta functions as well.  In \cite{SV}, the authors used this idea to prove the following result.

\begin{thm}[\cite{SV}, Thm.~2]\label{thm:p=2-codes-theorem}
Let $p=2$ and $C$ be a code of size $n$ defined over $\R$ and $\Th_{\La_\ell} (C)$ be its corresponding theta
function for level $\ell$. Then the following hold:
\begin{description}
\item [i)]
For $\ell < \frac {2(n+1)(n+2)}{n} -1$ there is a $\delta$-dimensional family of symmetrized weight
enumerator polynomials corresponding to $\Th_{\La_\ell}(C)$, where\\ $\delta \geq \frac
{(n+1)(n+2)}{2}-\frac{n(\ell+1)}{4} - 1$.

\item [ii)]
For $\ell \geq \frac {2(n+1)(n+2)}{n} -1$ and $n < \frac{\ell+1}{4}$ there is a unique symmetrized weight
enumerator polynomial which corresponds to $\Th_{\La_\ell}(C)$.
\end{description}
\end{thm}

Thus, for $\ell$ large enough, non-equivalent codes have different theta functions.  [Insert table, like the table below, showing results of this theorem.]

The method of Theorem~\ref{thm:p=2-codes-theorem} involved looking at the first positive power of $q$ in each theta series.  We can improve on the above theorem in certain cases.  Let $f(x,y,z)$ be a homogeneous polynomial in three variables of degree $n$.  Write $$f(x,y,z)=\sum_{i+j+k=n}c_{ijk}x^iy^jz^k,$$ a polynomial with ${n+2\choose2}=(n+2)(n+1)/2$ coefficients.  For any level $\ell$, one can calculate the theta function associated to each monomial $x^iy^jz^k$, represent that power series as a vector of coefficients (truncated at some degree), create the matrix $M_{\ell,n}$ formed by such vectors, and then compute the dimension of its nullspace.  If $\dim(\text{Nul}(M_{\ell,n}))=0$, then there do not exist two homogeneous polynomials of degree $n$ with the same theta series for level $\ell$, and so there therefore do not exist two codes of length $n$ with the same theta series for level $\ell$.

For example, in the case of $p=2$, $n=4$, there are $15$ monomials.  We compute the coefficients of the theta series associated to each of these monomials for level $\ell=15$ up to the $q^{16}$ term.  The matrix $M_{15,4}$ is given by
\begin{center}$M_{15,4}=\left(
\begin{array}{ccccccccccccccc}
1&0&0&0&0&0&0&0&0&0&0&0&0&0&0\\
0&0&2&0&0&0&0&0&0&0&0&0&0&0&0\\
0&0&0&0&0&4&0&0&0&0&0&0&0&0&0\\
0&0&0&0&0&0&0&0&0&8&0&0&0&0&0\\
8&2&0&0&0&0&0&0&0&0&0&0&0&0&16\\
0&0&12&0&4&0&0&0&0&0&0&0&0&0&0\\
0&2&0&0&0&16&0&0&8&0&0&0&0&0&0\\
0&0&0&0&4&0&0&0&0&16&0&0&0&16&0\\
24&12&0&4&0&0&0&0&8&0&0&0&0&0&0\\
0&0&26&0&16&0&0&8&0&0&0&0&0&16&0\\
0&14&0&8&0&24&0&0&16&0&0&0&16&0&0\\
0&0&0&0&20&0&0&16&0&24&0&0&0&0&0\\
32&24&0&20&0&0&8&0&24&0&0&0&32&0&64\\
0&0&28&0&20&0&0&24&0&0&0&16&0&16&0\\
0&36&0&40&0&32&24&0&16&0&0&0&16&0&0\\
0&0&2&0&36&0&0&48&0&48&0&48&0&48&0\\
40&18&0&40&0&8&40&0&32&0&16&0&32&0&0\\
\end{array}
\right)$\end{center}
This $17\times15$ matrix has rank 15 and therefore nullspace dimension 0.  We conclude that there cannot be two non-equivalent codes of length $n=4$ and the same theta series for level $\ell=15$.

In the table below, we calculate the nullity of $M_{\ell,n}$ for $\ell\in\{3,7,11,\dots,35\}$ and $n\in\{1,2,\dots,12\}$.  In particular, note that each ``$0$'' in the table describes a situation in which we cannot have two non-equivalent codes with the same theta series.

\begin{center}
$\begin{array}{c|ccccccccc} & \ell=3 &   \ell=7 & \ell=11 & \ell=15 & \ell=19   & \ell=23  &  \ell=27   &  \ell=31  &   \ell=35 \\ \hline
   n=1 &  1 & 0 & 0 &   0 &   0 &   0 &   0 &    0 &    0 \\
   n=2 &  3 &  1 &   0 &   0 &   0 &   0 &   0 &    0 &    0 \\
   n=3 &  6 &  3 &  1 &   0 &   0 &   0 &   0 &    0 &    0 \\
   n=4 &  10 &  6 &  3 &   0 &   0 &   0 &   0 &    0 &    0 \\
   n=5 &  15 &  10 &  6 &   0 &  1 &   0 &   0 &    0 &    0 \\
   n=6 &  21 &  15 &  10 &  1 &  3 &  1 &   0 &    0 &    0 \\
   n=7 &  28 &  21 &  15 &  3 &  6 &  3 &   0 &    0 &    0 \\
   n=8 &  36 &  28 &  21 &  6 &  10 &  6 &   0 &   1 &    0 \\
   n=9 &  45 &  36 &  28 &  10 &  15 &  10 &   1 &  3 &    0 \\
   n=10 & 55 &  45 &  36 &  15 &  21 &  15 &   3 &  6 &    0 \\
   n=11 & 66 &  55 &  45 &  21 &  28 &  21 &   6 &  10 &    0\\
   n=12 & 78 &  66 &  55 &  28 &  36 &  28 &   10 &  15 &   1 \end{array}$
\end{center}

Note that if there is a dependence relation among polynomials with a certain value of $\ell$ and $n$, then there are dependence relations for the same level $\ell$ and all degrees $N\geq n$ obtained by multiplying the original relation by any polynomial of degree $(N-n)$.  Interestingly, this is not true for increasing values of $\ell$.

\subsubsection{From polynomials to codes}\label{subsec:p=2-polys-to-codes}
In the table above, we consider polynomials of degree $n$ and theta functions of level $\ell$.  In practice, when dealing with codes, we are only interested in polynomials that arise as weight enumerator polynomials.  This greatly limits our search.

Motivated by Example~\ref{ex:chua-length3-codes}, for $p=2$, we fixed values of $\ell$ and $n$ and implemented algorithms using the Sage computer algebra system to look at codes of the form $$C(a_1,a_2,\mathbf{v}) := a_1\langle\mathbf{v}\rangle + a_2\langle\mathbf{v}\rangle^\perp$$ for $a_1,a_2\in\mathcal{R}$ and $\mathbf{v}\in\mathcal{R}^n$.  Note that if $\R$ is a field, then this is rather straightforward.  If $\R$ is not a field, then one must take extra care to compute the dual space because of the presence of zero divisors.

For small $n$ and for any $\ell$ we can compute all of the codes of this form, their weight enumerator polynomials, and their corresponding theta functions.  The number of different weight enumerator polynomials, followed by the number of theta functions for each combination of $n$ and $\ell$, is given in the table below.

\begin{center}
$\begin{array}{|c||c|c|c|c|}
\hline
{\bf p=2}
& \ell=3 & \ell=7 & \ell=11 & \ell=15  \\ \hline \hline
n=2 & 2,2 & 5,4 & 2,2 & 5,5 \\ \hline
n=3 & 3,3 & 11,8 & 3,3 & 11,11 \\ \hline
n=4 & 5,4 & 14,13 & 5,5 & 14,14 \\ \hline
n=5 & 6,5 & 18, 17 & 6,6 & 18,18 \\ \hline
\end{array}$ \\ 
\end{center}

One can then see that in the cases of $\ell=7$ and $n\geq 2$, as well as $\ell=3$ and $n\geq4$, we have non-equivalent codes with the same theta function.  The remainder of this subsection contains the specific examples.

\subsubsection{Specific examples of pairs of codes with different theta series for $p=2$}
In the following example, we give a pair of non-equivalent codes with $n=2$ and the same theta function for $\ell=7$.

\begin{exmp}\label{ex:length2codes}
Let $p=2$ and $\ell\equiv7\text{ (mod }8)$. 
Consider the following codes of length $n=2$.
$$C_1=C(\omega, \omega+1, (1,1)), \, C_2=C(\omega,\omega+1,(0,1))$$

The codes, along with their weight enumerator polynomials and theta series for various levels $\ell$ are given in Table \ref{tab:example_p=2,n=2,ell=7}.  As in the motivating example given earlier, these codes, which have different weight enumerator polynomials, have the same theta series for $\ell=7$ and different theta series for $\ell>7$.\end{exmp}

One can see that these codes have a very similar form to the motivating example in Example~\ref{ex:chua-length3-codes}.  In fact, they lead to a family of examples as follows.
\begin{thm}\label{thm:family-of-pairs-of-codes}
For $p=2$, $\ell\equiv 7\text{ (mod }8)$, and any $n\geq 2$, consider the length $n$ codes 
$$C_{n,1} = C(\omega,\omega+1,(0,\dots,0,1,1))$$
$$C_{n,2} = C(\omega,\omega+1,(0,\dots,0,0,1))$$
These codes, which have different symmetric weight enumerator polynomials, have the same theta series for $\ell=7$ and different theta series for $\ell>7$.
\end{thm}

\begin{proof}For $n>2$ and $i\in\{1,2\},$ one sees that $C_{n,i}$ is the concatenation of the length 1 code $\{(0),(\omega+1)\}$ and the length $(n-1)$ code $C_{n-1,i}$.  The symmetric weight enumerator of this length 1 code is $(X+Z)$.  Thus, for $n\geq2$,  $$swe_{C_{n,1}}=(X^2+Y^2+2Z^2)(X+Z)^{n-2} \text{ and } swe_{C_{n,2}}=(X+Z)^n.$$

To obtain the theta series for $C_{n,1}$ and $C_{n,2}$, one can take the theta series for $(X^2+Y^2+2Z^2)$ and $(X+Z)^2$ and multiply each by the theta series for $(X+Z)^{n-2}$.  Since the theta series for the polynomials $(X^2+Y^2+2Z^2)$ and $(X+Z)^2$ are equal for $\ell=7$, the theta series for these codes $C_{n,1}$ and $C_{n,2}$ are also equal for any $n\geq2$.  Also, since the theta series for these polynomials are not equal for $\ell>7$, the theta series for $C_{n,1}$ and $C_{n,2}$ are not equal for $\ell>7$.\end{proof}
\begin{rem}There are three pairs of non-equivalent codes with the same theta function in the case of $n=3$ and $\ell=7$.  The pair of codes in Example~\ref{ex:chua-length3-codes} is one such pair, given by the above theorem with $n=3$.  A second pair of codes, given in Example~\ref{ex:concat_example_p=2,n=3,ell=7}, is also seen to be an extension of Example~\ref{ex:length2codes}.  A third pair of codes, given in Example~\ref{ex:new-length3-codes}, is of a different form.
\end{rem}


\begin{exmp}\label{ex:concat_example_p=2,n=3,ell=7}
Let $p=2$ and $\ell\equiv 7\text{ (mod }8)$.  Consider the following codes of length $n=3$.
$$C_1=C(\omega,1,(0,1,1)),\, C_2=C(1,\omega,(0,0,1))$$

The codes, along with their weight enumerator polynomials and theta series for various levels $\ell$ are given in Table \ref{tab:concat_example_p=2,n=3,ell=7}.  These codes, which have different weight enumerator polynomials, have the same theta series for $\ell=7$ and different theta series for $\ell>7$.

Note that $swe_{C_1}=(X^2+Y^2+2Z^2)(X+Y+2Z)$ and $swe_{C_2}=(X+Z)^2(X+Y+2Z)$.  $C_1$ is really just the concatenation of the code $\{(0),(1),(\omega),(\omega+1)\}$ with the code $\{(0,0),(1,1),(\omega,\omega),(\omega+1,\omega+1)\}$.  And $C_2$ is the concatenation of the code $\{(0,0),(\omega,0),(0,\omega),(\omega,\omega)\}$ with the code $\{(0),(1),(\omega),(\omega+1)\}$.
\end{exmp}

In the next example, we have a pair of codes for $p=2$ and $n=3$ which do not originate from the codes in Example~\ref{ex:length2codes}.

\begin{exmp}\label{ex:new-length3-codes}
Let $p=2$ and $\ell\equiv 7\text{ (mod }8)$.  Consider the following codes of length $n=3$.
$$C_1 = C(1,\omega,(1,1,1)),\,C_2=C(\omega,1,(1,1,\omega+1))$$

The codes, along with their weight enumerator polynomials and theta series for various levels $\ell$ are given in Table \ref{tab:new_example_p=2,n=3,ell=7}.  These codes, which have different weight enumerator polynomials, have the same theta series for $\ell=7$ and different theta series for $\ell>7$.

Further, note that these codes are not concatenations of pairs of shorter codes.  This follows from the fact that their symmetric weight enumerator polynomials cannot be factored into a product of polynomials with non-negative coefficients.  Hence, this pair of codes is fundamentally different from the pairs given earlier in this section.
\end{exmp}

Finally, we present an example of two codes with $\ell\equiv 3\text{ (mod }8)$, meaning the codes are defined over the field $\mathbb{F}_4$.
\begin{exmp}\label{ex:p=2,n=4,ell=3}
Let $p=2$ and $\ell\equiv 3\text{ (mod }8)$.  Consider the following codes of length $n=4$.
$$C_1 = C(1,1,(\omega,\omega,\omega,\omega)),\,C_2=C(1,1,(\omega,\omega,1,1))$$

The codes, along with their weight enumerator polynomials and theta series for various levels $\ell$ are given in Table \ref{tab:p=2,n=4,ell=3}.  These codes, which have different weight enumerator polynomials, have the same theta series for $\ell=3$ and different theta series for $\ell>3$.\end{exmp}

For $\ell=3$, the theta series associated to the monomials $Y$ (which corresponds to a 1 in a codeword) and $Z$ (which corresponds to a $\omega$ or $\omega+1$) are the same.  For the above example, note that, although these two codes have different weight enumerator polynomials, they have the same distribution of weights.  Still, this is our first example of two non-equivalent codes defined over a field with the same theta series.

\subsection{Toward examples for $p=3$}\label{sec:p=3} 
For $p=3$, we proceed as we did in Section~\ref{subsec:p=2} for $p=2$.  For a fixed level $\ell$ with $3\nmid\ell$, let $C$ be a code over $\R$ of length $n$.  The symmetric weight enumerator of $C$ is a polynomial of 4 variables $x,y,z,w$.  There are ${n+3\choose3}=(n+3)(n+2)(n+1)/6$ monomials of degree $n$ in 4 variables.  For the level $\ell$, calculate the theta function associated to each monomial $x^iy^jz^kw^l$, represent that power series as a vector of coefficients (truncated at some degree), create the matrix $M_{\ell,n}$ formed by such vectors, and then compute the dimension of its nullspace.

In the table below, the nullity of $M_{\ell,n}$ is given for $\ell\in\{7,11,19,23,31,35,43,47\}$ and $n\in\{1,2,3,4,5,6,7,8,9\}$.  In particular, note that each ``$0$'' in the table describes a situation in which we cannot have two non-equivalent codes with the same theta series.

\begin{center}
$\begin{array}{c|cccccccc}
    &          \ell=7  &  \ell=11   &     \ell=19   & \ell=23  &     \ell=31  &  \ell=35  &   \ell=43 &  \ell=47  \\ \hline \hline
   n=1  &  0 &   0 &      0 &   0 &     0 &   0 &     0 &   0 \\   
   n=2  &  1 &  0 &      0 &   0 &   0 &   0 &     0 &   0 \\  
   n=3  &  4 &  1 &     0 &   0 &   0 &   0 &     0 &   0 \\  
   n=4  &  11 &  5 &     0 &   0 &  0 &   0 &    0 &   0 \\ 
   n=5  &  24 &  14 &     0 &   0 &   0 &   0 & 0 &   0 \\
   n=6  &  44 &  30 &    4 &  2 &   0 &   0 &   0 &   0 \\ 
   n=7  &  72 &  54 &    16 &  9 &  0 &   0 &  0 &   0 \\
   n=8  &  109 & 87   & 38 &  25 &  5 &  2 & 1 & 2 \\
   n=9  &  156 & 130 & 72 & 53 & 20 & 8 & 4 & 8 \end{array}$\end{center}
The situation here is quite different from the $p=2$ case where for each length $n$ there are levels $\ell$ such that the nullity is zero.  In particular, for $p=3$ and $n=8$, there is a polynomial dependence relation for all levels $\ell$ given by
\begin{center}$\begin{array}{l}  2Y^8 + 2Y^2Z^6 + X^5YZW + X^2Y^4ZW + X^4Z^2W^2 + 10XY^3Z^2W^2 + 2Y^2W^6 \\ =
X^3Y^5 + 2X^3Y^2Z^3 + 4Y^5Z^3 + X^2YZ^4W + XZ^5W^2 + 2X^3Y^2W^3 + 4Y^5W^3 \\ + 2Y^2Z^3W^3 +X^2YZW^4 + XZ^2W^5.\end{array}$\end{center}
Thus, if one were to find two codes of length at least 8 defined over $\R=\mathbb{F}_3\times\mathbb{F}_3$ or $\mathbb{F}_9$ such that the difference of their symmetric weight enumerator polynomials is a multiple of the difference of these polynomials, then one would have a pair of codes with different weight enumerator polynomials and the same theta function for all levels $\ell$.

\subsubsection{From polynomials to codes for $p=3$}  As in Section~\ref{subsec:p=2-polys-to-codes}, we use our implementation in Sage for $p=3$ to search over codes of the form $$C(a_1,a_2,\mathbf{v}):=a_1\langle\mathbf{v}\rangle + a_2\langle\mathbf{v}\rangle^\perp$$ for $a_1,a_2\in \R$ and $\mathbf{v}\in\R^n$.  Note that for the dual space, one needs to know how complex conjugation behaves.  In $\O_K$, $\omega_\ell = \frac{-1+\sqrt{-\ell}}{2}$, so $\overline\omega_\ell = \frac{-1-\sqrt{-\ell}}{2}$.  Thus, $\omega_\ell + \overline{\omega_\ell} = -1$, so we have that $\overline{a+b\omega_\ell}=a-b-b\omega_\ell$.  Complex conjugation therefore behaves the same way in $\mathcal{R}$.

For small $n$ and any $\ell$, we can compute all codes of this form, their weight enumerator polynomials, and their corresponding theta functions.  The number of different weight enumerator polynomials, followed by the number of theta functions for each combination of $n$ and $\ell$, is given in the table below.

\begin{center}$\begin{array}{|c|c|c|c|c|}
\hline
p=3
& \ell=7 & \ell=11 & \ell=19 & \ell=23  \\ \hline
n=2 & 2,2 & 9,9 & 2,2 & 9,9 \\ \hline  
n=3 & 4,4 & 25,25 & 4,4 & 25,25 \\ \hline  
\end{array}$ 
\end{center}

As one can see, in this range there are no pairs of codes with different weight enumerator polynomials and the same theta function.  However, as we saw earlier, there are dependence relations between the monomials of degree 2 and 3 for $p=3$.  The problem is that the codes, being $\R$-submodules, have a lot of structure and thus lead to relatively few weight enumerator polynomials.

Below, we modify our search criteria slightly.

\subsubsection{Modified method - ``$\mathbb{F}_p$-submodule codes''}
If we loosen one restriction on our codes - namely, treating them as $\mathbb{F}_p$-submodules rather than as $\R$-submodules -- then we do find dependence relations among the resulting polynomials.  For $p$ prime, we will refer to these as ``$\mathbb{F}_p$-submodule codes.''

Note that these are no longer ``codes'' under the definition given in Section~\ref{sec:theta_functions_of_codes}, though any code under the standard definition is necessarily a $\mathbb{F}_p$-submodule code as well.  As we will see below, there do exist pairs of these structures with the same theta function.  One then has the following question.
\bigskip

{\bf Question:} For some prime $p$ and level $\ell$, given a pair of length $n$ non-equivalent ``$\mathbb{F}_p$-submodule codes,'' can one find a way to create a pair of codes with the same theta function?

\bigskip
To be precise, we consider $\mathbb{F}_p$-submodule codes of the form $$C(a_1,a_2,\mathbf{v}):=a_1\langle\mathbf{v}\rangle + a_2\langle\mathbf{v}\rangle^\perp$$ for $a_1,a_2\in \R$ and $\mathbf{v}\in\R^n$, where we let $\langle \mathbf{v}\rangle := \{c\mathbf{v} : c\in\mathbb{F}_p\}$ rather than considering all $c\in\mathcal{R}$.

For small $n$ and any $\ell$, we can compute all $\mathbb{F}_3$-submodule codes of this form, their weight enumerator polynomials, and their corresponding theta functions.  The number of different weight enumerator polynomials, followed by the number of theta functions for each combination of $n$ and $\ell$, is given in the table below.

\begin{center}$\begin{array}{|c|c|c|c|c|}
\hline
p=3, \mathbb{F}_3\text{-submodule codes} 
& \ell=7 & \ell=11 & \ell=19 & \ell=23  \\ \hline 
n=2 & 12,12 & 17,17 & 12,12 & 17,17 \\ \hline 
n=3 & 147,144 & 71,70 & 147,147 & 71,71 \\ \hline   
\end{array}$ 
\end{center}

For $n=3$, there are pairs of non-equivalent $\mathbb{F}_3$-submodule codes with the same theta function for $\ell=7$ and $\ell=11$.  These examples are presented below.  Following the notation of Example~\ref{ex:p=3-swe}, the symmetric weight enumerator polynomials have 4 variables, denoted $X,Y,Z,W$.

\begin{exmp}
Let $p=3$ and $\ell\equiv 7\text{ (mod }12)$.  Consider the following pairs of $\mathbb{F}_3$-submodule codes of length $n=3$.
$$C_{1,1} = C(\omega+1,1,(1,\omega,\omega+1)),\, C_{1,2} = C(\omega,\omega,(\omega+1,\omega+1,\omega+2)).$$ 
$$C_{2,1} = C(\omega,\omega+1,(\omega,\omega,\omega+2)), \, C_{2,2} = C(\omega,2\omega+1,(1,\omega+1,\omega+2)).$$ 
$$C_{3,1} = C(1,\omega,(1,\omega,\omega+1)), \, C_{3,2} = C(1, \omega, (1,1,\omega+2)).$$ 


These pairs of $\mathbb{F}_3$-submodule codes have the following symmetric weight enumerators and theta functions.
\begin{center}\begin{tabular}{rcl}
$swe_{C_{1,1}}$&$=$&$X^3 + 2Y^3 + 4XZ^2 + 4YZ^2 + 2Z^3 + 2XZW + $ \\ & & $ 6YZW + 2Z^2W + 2YW^2 + 2ZW^2.$ \\ \\
$swe_{C_{1,2}}$&$=$&$X^3 + 2XYZ + 2Y^2Z + 2XZ^2 + 4YZ^2 + 2Z^3 + $ \\ & & $ 2Y^2W + 2XZW + 4YZW + 2Z^2W + 4ZW^2,$ \\
$\th_{\Lambda_{7}(C_{1,1})}(q)$&$=$& $\th_{\Lambda_{7}(C_{1,2})}(q)$ \\
&$=$&$1 + 2q^3 + 4q^4 + 4q^5 + 12q^6 + 12q^7 + 8q^8 + 22q^9 + 42q^{10}+\dots$ \\ \\
$\th_{\Lambda_{19}(C_{1,1})}(q)$&$=$& $ 1 + 2q^3 + 6q^6 + 12q^9 + 4q^{10} +\dots$\\
$\th_{\Lambda_{19}(C_{1,2})}(q)$&$=$& $ 1 + 2q^6 + 2q^7 + 12q^9 + 6q^{10} +\dots$ 
\end{tabular}\end{center}

\begin{center}\begin{tabular}{rcl}
$swe_{C_{2,1}}$&$=$&$X^3 + 2XY^2 + 2XZ^2 + 4YZ^2 + 2Z^3 + 10YZW + $ \\ & & $4Z^2W + 2XW^2.$ \\
$swe_{C_{2,2}}$&$=$&$X^3 + 2X^2Z + 2Y^2Z + 4YZ^2 + 2Z^3 + 2XYW + $ \\ & & $8YZW + 4Z^2W + 2ZW^2.$ \\ \\
$\th_{\Lambda_{7}(C_{2,1})}(q)$&$=$& $\th_{\Lambda_{7}(C_{2,2})}(q)$ \\
&$=$&$1 + 2q^2 + 2q^4 + 8q^5 + 2q^6 + 20q^7 + 22q^8 + 6q^9 + 38q^{10}+\dots$ \\ \\
$\th_{\Lambda_{19}(C_{2,1})}(q)$&$=$& $ 1 + 2q^2 + 4q^5 + 2q^8 + 6q^9 + 2q^{10} +\dots$ \\
$\th_{\Lambda_{19}(C_{2,2})}(q)$&$=$& $ 1 + 2q^5 + 2q^7 + 4q^8 + 6q^9 + 4q^{10} +\dots$
\end{tabular}\end{center}

\begin{center}\begin{tabular}{rcl}
$swe_{C_{3,1}}$&$=$&$X^3 + 2XYZ + 2Y^2Z + 4XZ^2 + 2YZ^2 + 2Z^3 + 2Y^2W + $ \\ & & $6YZW + 4Z^2W + 2W^3
.$ \\
$swe_{C_{3,2}}$&$=$&$X^3 + 2XYZ + 4Y^2Z + 2XZ^2 + 2YZ^2 + 2Z^3 + 2XZW + $ \\ & & $4YZW + 4Z^2W + 2YW^2 + 2ZW^2.$ \\ \\
$\th_{\Lambda_{7}(C_{3,1})}(q)$&$=$& $\th_{\Lambda_{7}(C_{3,2})}(q)$ \\
&$=$&$1 + 2q^3 + 6q^4 + 2q^5 + 8q^6 + 16q^7 + 10q^8 + 20q^9 + 40q^{10}+\dots$ \\ \\
$\th_{\Lambda_{19}(C_{3,1})}(q)$&$=$& $ 1 + 2q^6 + 2q^7 + 12q^9 + 8q^{10} +\dots$ \\
$\th_{\Lambda_{19}(C_{3,2})}(q)$&$=$& $ 1 + 2q^6 + 4q^7 + 8q^9 + 10q^{10} +\dots$
\end{tabular}\end{center}

Note that for each pair $C_{i,1}$ and $C_{i,2}$, one has $swe_{C_{i,1}}\neq swe_{C_{i,2}}$, $\th_{\Lambda_{7}(C_{i,1})}(q)=\th_{\Lambda_{7}(C_{i,2})}(q)$, and $\th_{\Lambda_{7}(C_{i,1})}(q)\neq\th_{\Lambda_{7}(C_{i,2})}(q)$ for $\ell>7$.
\end{exmp}

\begin{exmp}
Let $p=3$ and $\ell\equiv 11\text{ (mod }12)$.  Consider the following $\mathbb{F}_3$-submodule codes of length $n=3$.
$$C_1 = C(\omega,2\omega+1,(0,1,1)),\, C_2=C(\omega,2\omega+1,(1,1,1)).$$

These codes have the following symmetric weight enumerator polynomials.
$$swe_{C_1} = X^3 + 2X^2Z + 4XZ^2 + 8Z^3 + 4XYW + 8YZW.$$
$$swe_{C_2} = X^3 + 2Y^3 + 6XZ^2 + 4Z^3 + 12YZW + 2W^3.$$

For $\ell=11$, we have 
$$\theta_{\Lambda_{11}(C_1)}(q)=\theta_{\Lambda_{11}(C_2)}(q)=1 + 2q^3 + 12q^6 + 40q^9 + 38q^{12} + 88q^{15} +\dots.$$
For $\ell>11$, $\theta_{\Lambda_\ell(C_1)}(q)\neq\theta_{\Lambda_\ell(C_2)}(q).$  In particular, 
$$\theta_{\Lambda_{23}(C_1)}(q)= 1 + 2q^6 + 14q^9 + 14q^{12} + 24q^{15} + \dots$$ and 
$$\theta_{\Lambda_{23}(C_2)}(q)= 1 + 2q^3 + 6q^6 + 12q^9 + 8q^{12} + 24q^{15} + \dots.$$
\end{exmp}

\subsubsection{The case of $p=5$}
Finally, we consider the case of $p=5$.  Working as in Section~\ref{sec:p=3}, we calculate the nullity of the matrix $M_{\ell,n}$ for small values of $\ell$ and $n$.

\begin{center}
$\begin{array}{c|cccccccc}
    &          \ell=3  &  \ell=7   &     \ell=11   & \ell=19  &     \ell=23  &  \ell=27  &   \ell=31 &  \ell=39  \\ \hline \hline
   n=1  &  4 &   0 &      0 &   0 &     0 &   0 &     0 &   0 \\   
   n=2  &  30 &  10 & 1 &   1 &   0 &   0 &     0 &   0 \\  
   n=3  &  131 &  91 &    51 & 19 &   1 &   2 &     1 &   1 \\  
  \end{array}$\end{center}

Searching becomes computationally more difficult as we consider larger primes. Because of this, for $p=5$, we slightly restrict our search space to begin with vectors in $\mathbb{F}_p^n$ rather than $\R^n$.  That is, we consider codes and $\mathbb{F}_5$-submodule codes of the form $a_1\langle\mathbf{v}\rangle+a_2\langle\mathbf{v}\rangle^\perp$ for $a_1,a_2\in\mathcal{R}$, $\mathbf{v}\in\mathbb{F}_p^n$.  The numbers of different symmetric weight enumerators, followed by the numbers of different theta series, are given in the tables below, first for the case of codes and second for the case of $\mathbb{F}_5$-submodule codes.

\begin{center}$\begin{array}{|c|c|c|c|c|c|}
\hline
p=5, \text{codes, using $\mathbf{v}\in\mathbb{F}_5^n$}
& \ell=3 & \ell=7 & \ell=11 & \ell=19 & \ell=23 \\ \hline 
n=1 & 1,1 & 1,1 & 3,3 & 3,3 & 1,1 \\ \hline
n=2 & 1,1 & 1,1 & 18,18 & 17,16 & 1,1  \\ \hline 
\end{array}$
\end{center}

\begin{center}$\begin{array}{|c|c|c|c|c|c|}
\hline
p=5, \mathbb{F}_5\text{-submodule codes, $\mathbf{v}\in\mathbb{F}_5^n$} 
& \ell=3 & \ell=7 & \ell=11 & \ell=19 & \ell=23  \\ \hline 
n=1 & 4,2 & 4,4 & 4,4 & 4,4 & 4,4 \\ \hline
n=2 & 72,20 & 72,71 & 72,71 & 59,58 & 72,72 \\ \hline 
\end{array}$
\end{center}

Note that, in the case of $n=2$ and $\ell=19$, we have two codes with different weight enumerator polynomials and the same theta function.  Also, for $n=2$ and $\ell=3,7,11,19$, we have pairs of $\mathbb{F}_5$-submodule codes with different weight enumerator polynomials and the same theta function.

When $p=5$, by Lemma~\ref{lemma-congruences}, the symmetric weight enumerator polynomial has $(5+1)^2/4=9$ variables, denoted $X_1,X_2,\dots,X_9$.
In particular, \\ $swe_C(X_1, X_2, X_3, X_4, X_5, X_6, X_7, X_8, X_9)= $
\begin{center}\begin{tabular}{l} $cwe_C(X_1,X_2,X_6,X_6,X_2,X_3,X_4,X_7,X_4,X_3,X_5,X_8,$ \\ \hspace{.4in}$X_8,X_5,X_9,X_5,X_9,X_5,X_8,X_8,X_3,X_3,X_4,X_7,X_4).$\end{tabular}\end{center}  This follows analogously to Example~\ref{ex:p=3-swe} where the case of $p=3$ is considered.  To evaluate the theta function, let $X_1=\th_{\La_{0,0}}(q), X_2=\th_{\La_{1,0}}(q), X_3=\th_{\La_{0,1}}(q), X_4=\th_{\La_{1,1}}(q), X_5=\th_{\La_{0,2}}(q), X_6=\th_{\La_{2,0}}(q), X_7=\th_{\La_{2,1}}(q), X_8=\th_{\La_{1,2}}(q), X_9=\th_{\La_{1,3}}(q).$

Note that in the table of $\mathbb{F}_5$-submodule codes for $n=2$ and $\ell=3$, there are 72 polynomials giving rise to just 20 theta functions.  This is because when $\ell=3$, there are dependence relations between the theta functions.  In particular, $\th_{\La_{1,0}}(q)=\th_{\La_{0,1}}(q)$, $\th_{\La_{2,0}}(q)=\th_{\La_{0,2}}(q)$, $\th_{\La_{1,1}}(q)=\th_{\La_{1,3}}(q)$, and $\th_{\La_{2,1}}(q)=\th_{\La_{1,2}}(q)$.  The examples below are for levels $\ell>3$, for which these theta functions are linearly independent.

First is an example of two non-equivalent codes defined over $\mathbb{F}_5\times\mathbb{F}_5$ with the same theta function for $\ell=19$.

\begin{exmp}\label{ex:p=5-codes-same-theta}Let $p=5$ and $\ell\equiv19\text{ (mod }20)$.  Consider the following codes of length $n=2$ defined over $\R=\mathbb{F}_5\times\mathbb{F}_5$.
$$C_1 = C(\omega,\omega,(0,1)),\, C_2 = C(\omega,\omega+1,(1,2)).$$
These codes have the following symmetric weight enumerator polynomials.

$$swe_{C_1} = X_1^2 + 4X_1X_3 + 4X_3^2 + 4X_1X_5 + 8X_3X_5 + 4X_5^2.$$
$$swe_{C_2} = X_1^2 + 8X_3X_5 + 4X_2X_6 + 8X_4X_8 + 4X_7X_9.$$

For $\ell=19$ we have $$\theta_{\lambda_{19}(C_1)}(q) = \theta_{\lambda_{19}(C_2)}(q)=1 + 4q^5 + 4q^{10} + 4q^{20} + 16q^{25} + 16q^{30} + 8q^{35} +\dots.$$
For $\ell>19$, $\theta_{\lambda_{19}(C_1)}(q) \neq \theta_{\lambda_{19}(C_2)}(q)$.  In particular, 
$$\theta_{\Lambda_{39}(C_1)}(q)=1 + 4q^{10} + 4q^{20} + 4q^{25} + 4q^{30} + 8q^{35} + 16q^{40} +\dots$$ and 
$$\theta_{\Lambda_{39}(C_2)}(q)= 1 + 4q^5 + 4q^{10} + 4q^{20} + 8q^{25} + 4q^{40} + 4q^{45}+ \dots.$$
\end{exmp}

Note that the two codes in the example above have different weight distributions of codewords.  The first code has eight codewords of weight 1, while the second code has none.

Finally, we give an example of two non-equivalent $\mathbb{F}_5$-submodule codes defined over $\mathbb{F}_5\times\mathbb{F}_5$ with the same theta function for $\ell=11$.

\begin{exmp}
Let $p=5$ and $\ell\equiv 11\text{ (mod }20)$.  Consider the following $\mathbb{F}_5$-submodule codes of length $n=2$ defined over $\R=\mathbb{F}_5\times\mathbb{F}_5$.
$$C_1 = C(1,3\omega+1,(1,3)),\, C_2=C(\omega+1,\omega+1,(0,1)).$$  We have the following symmetric weight enumerator polynomials.

$$swe_{C_1} = X_1^2 + 8X_3X_5 + 4X_2X_6 + 8X_4X_8 + 4X_7X_9.$$
$$swe_{C_2} = X_1^2 + 4X_1X_4 + 4X_4^2 + 4X_1X_8 + 8X_4X_8 + 4X_8^2.$$

For $\ell=11$, we have 
$$\theta_{\lambda_{11}(C_1)}(q)=\theta_{\lambda_{11}(C_2)}(q)=1 + 4q^5 + 4q^{10} + 8q^{15} + 20q^{20} + 16q^{25}+\dots.$$  
However, for $\ell>11$, $\theta_{\Lambda_\ell(C_1)}(q)\neq\theta_{\Lambda_\ell(C_2)}(q).$  In particular, 
$$\theta_{\Lambda_{31}(C_1)}(q)= 1 + 4q^5 + 4q^{10} + 4q^{20} + 8q^{25} +\dots$$ and 
$$\theta_{\Lambda_{31}(C_2)}(q)= 1 + 4q^{10} + 8q^{20} + 4q^{25} + \dots.$$
\end{exmp}

\section{Concluding remarks}
The examples presented in this paper are for $p=2,3,5$.  One can work with larger primes, though the number of computations grows in each case.  Based on the results here as well as some limited computations with $p=7$, it seems reasonable to expect that for any prime $p$, there will be pairs of $\mathbb{F}_p$-submodule codes with the same theta function.  It would then be interesting to see if there is a way to take such pairs to create pairs of actual codes with the same theta function.

\section*{Appendix - Full details for given examples with $p=2$} 
In this appendix, we give the specific data for the examples of pairs of non-equivalent codes with the same theta function for $p=2$.  Tables \ref{tab:chua-example}-\ref{tab:p=2,n=4,ell=3} contain the codes given in Examples \ref{ex:chua-length3-codes}-\ref{ex:p=2,n=4,ell=3}.

\begin{table}
\begin{tabular}{|l|l|}
\hline & \\
$\begin{array}{ll}C_1=& \omega\langle(0,1,1)\rangle+(\omega+1)\langle(0,1,1)\rangle^\perp\end{array}$ 
&
$\begin{array}{ll}C_2 = \omega\langle(0,0,1)\rangle+(\omega+1)\langle(0,0,1)\rangle^\perp\end{array}$
\\ & \\ 
$\begin{array}{ll} C_1=&\left\{(0,0,0), (0,\omega,\omega), (\omega+1,0,0), \right. \\ & (\omega+1,\omega,\omega), (0,\omega+1,\omega+1), \\ & (0,1,1), (\omega+1,\omega+1,\omega+1), \\ & \left. (\omega+1,1,1)\right\}\end{array}$
&
$\begin{array}{ll} C_2=&\left\{(0,0,0), (0,0,\omega), (\omega+1,0,0), \right. \\ & (\omega+1, 0, \omega), (0,\omega+1,0), \\ & (0,\omega+1,\omega), (\omega+1, \omega+1, 0), \\ & \left. (\omega+1, \omega+1, \omega)\right\}.\end{array}$
\\ & \\
$\begin{array}{lcl}swe_{C_1}&=&X^3+X^2Z+XY^2+\\ & & 2XZ^2+Y^2Z+2Z^3 \\ &=& (X^2+Y^2+2Z^2)(X+Z)\end{array}$
&
$\begin{array}{lcl}swe_{C_2}&=&X^3+3X^2Z+ 3XZ^2+Z^3 \\ &=&(X+Z)^3. \end{array}$
\\ & \\ 
\begin{tabular}{ll} $\theta_{\Lambda_7(C_1)}(q) =$ & $ 1 + 6q^2 + 24q^4 + 56q^6 + $ \\ & $ 114q^8 + 168q^{10}+ \dots$\end{tabular}
& 
\begin{tabular}{ll} $\theta_{\Lambda_7(C_2)}(q) =$ & $ 1 + 6q^2 + 24q^4 + 56q^6 + $ \\ & $ 114q^8 + 168q^{10}+ \dots$\end{tabular}
\\ & \\
\begin{tabular}{ll} $\theta_{\Lambda_{15}(C_1)}(q) =$ & $ 1 + 4q^2 + 8q^4 + 18q^6 + $ \\ & $ 36q^8 + 34q^{10} + \dots$\end{tabular}
&
\begin{tabular}{ll} $\theta_{\Lambda_{15}(C_2)}(q) =$ & $ 1 + 12q^4 + 6q^6 + 48q^8 + $ \\ & $54q^{10}+ \dots$\end{tabular}
\\ & \\ \hline
\end{tabular}\caption{Two nonequivalent codes of length $n=3$ over $\mathbb{F}_2\times\mathbb{F}_2$ with the same theta series for $\ell=7$ and different theta series for $\ell>7$.  As given in Example~\ref{ex:chua-length3-codes}.}\label{tab:chua-example}\end{table}

\begin{table}\begin{tabular}{|l|l|}
\hline & \\
$\begin{array}{lcl}C_1 &= &\omega\langle(1,1)\rangle + (\omega+1)\langle(1,1)\rangle^\perp\end{array}$ 
&
$\begin{array}{lcl}C_2 &= & \omega\langle(0,1)\rangle + (\omega+1)\langle(0,1)\rangle^\perp\end{array}$
\\ & \\ 
$\begin{array}{lcl}C_1 &=& \{(0,0), (1,1), (\omega, \omega), \\ & & (\omega +1, \omega + 1)\}\end{array}$
&
$\begin{array}{lcl}C_2 &=& \{(0,0), (0, \omega), (\omega+1, 0),\\ & &(\omega+1, \omega)\}\end{array}$
\\ & \\
$\begin{array}{lcl}swe_{C_1} &= & X^2 + Y^2 + 2Z^2\end{array}$
&
$\begin{array}{lcl}swe_{C_2} &=& X^2 + 2XZ + Z^2\\ & =& (X+Z)^2\end{array}$
\\ & \\ 
$\begin{array}{ll} \theta_{\Lambda_7(C_1)}(q) = & 1 + 4q^2 + 12q^4 + 16q^6 + \\ & 28q^8 + 24q^{10} + 48q^{12} + \dots\end{array}$
& 
$\begin{array}{ll} \theta_{\Lambda_7(C_2)}(q)  = & 1 + 4q^2 + 12q^4 + 16q^6 + \\ & 28q^8 + 24q^{10} + 48q^{12} + \dots\end{array}$
\\ & \\
$\begin{array}{ll} \theta_{\Lambda_{15}(C_1)}(q) = & 1 + 4q^2 + 4q^4 + 12q^8 + \\ & 24q^{10} + 8q^{12} + \dots\end{array}$
&
$\begin{array}{ll} \theta_{\Lambda_{15}(C_2)}(q)  = & 1 + 8q^4 + 4q^6 + 16q^8 + \\ & 20q^{10} + 4q^{12} + \dots\end{array}$
\\ & \\ \hline
\end{tabular}\caption{Two nonequivalent codes of length $n=2$ over $\mathbb{F}_2\times\mathbb{F}_2$ with the same theta series for $\ell=7$ and different theta series for $\ell>7$.  As given in Example~\ref{ex:length2codes}.}\label{tab:example_p=2,n=2,ell=7}\end{table}

\begin{table}\begin{tabular}{|l|l|}
\hline & \\
$\begin{array}{lcl}C_1 &= &\omega\langle(0,1,1)\rangle + 1\langle(0,1,1)\rangle^\perp\end{array}$ 
&
$\begin{array}{lcl}C_2 &= & 1 \langle(0,0,1)\rangle + \omega \langle(0,0,1)\rangle^\perp\end{array}$
\\ & \\ 
$\begin{array}{lcl}C_1 &=& \{(0, 0, 0), (0, \omega, \omega), (0, 1, 1), \\ & &(0, \omega + 1, \omega + 1), (\omega, 0, 0), \\ & &(\omega, \omega, \omega), (\omega, 1, 1), \\ & &(\omega, \omega + 1, \omega + 1), (1, 0, 0), \\ & &(1, \omega, \omega), (1, 1, 1), \\ & & (1, \omega + 1, \omega + 1), (\omega + 1, 0, 0),\\ & & (\omega + 1, \omega, \omega), (\omega + 1, 1, 1),\\ & & (\omega + 1, \omega + 1, \omega + 1)\}\end{array}$
&
$\begin{array}{lcl}C_2 &=&  \{(0, 0, 0), (0, 0, \omega), (0, 0, 1), \\ & & (0, 0, \omega + 1), (\omega, 0, 0), (\omega, 0, \omega), \\ & & 
(\omega, 0, 1), (\omega, 0, \omega + 1), (0, \omega, 0),\\ & &  (0, \omega, \omega), (0, \omega, 1), (0, \omega, \omega +1),\\ & &  (\omega, \omega, 0), (\omega, \omega, \omega), (\omega, \omega, 1), \\ & & (\omega, \omega, \omega + 1)\}\end{array}$
\\ & \\
$\begin{array}{lcl}swe_{C_1} &= & X^3 + X^2Y + XY^2 + \\ & & Y^3 + 2X^2Z + 2Y^2Z + \\ & & 2XZ^2 + 2YZ^2 + 4Z^3\end{array}$
&
$\begin{array}{lcl}swe_{C_2} &=& X^3 + X^2Y + 4X^2Z + \\ & & 2XYZ + 5XZ^2 + YZ^2 + \\ & & 2Z^3\end{array}$
\\ & \\ 
$\begin{array}{ll} \theta_{\Lambda_7(C_1)}(q) = & 1 + 2q + 8q^2 + 8q^3 + \\ & 34q^4 + 24q^5 + 88q^6 +\\ &  34q^7 + 172q^8+\dots\end{array}$
& 
$\begin{array}{ll} \theta_{\Lambda_7(C_2)}(q) = & 1 + 2q + 8q^2 + 8q^3 + \\ & 34q^4 + 24q^5 + 88q^6 +\\ &  34q^7 + 172q^8+\dots\end{array}$
\\ & \\
$\begin{array}{ll} \theta_{\Lambda_{15}(C_1)}(q) = & 1 + 2q + 4q^2 + 8q^3 +\\ & 10q^4 + 8q^5 + 28q^6 +\\ & 52q^8 + \dots\end{array}$
&
$\begin{array}{ll} \theta_{\Lambda_{15}(C_2)}(q)  = & 1 + 2q + 14q^4 + 16q^5 + \\ & 8q^6 + 8q^7 + 64q^8+ \dots\end{array}$
\\ & \\ \hline
\end{tabular}\caption{Two nonequivalent codes of length $n=3$ over $\mathbb{F}_2\times\mathbb{F}_2$ with the same theta series for $\ell=7$ and different theta series for $\ell>7$.  As given in Example~\ref{ex:concat_example_p=2,n=3,ell=7}.}\label{tab:concat_example_p=2,n=3,ell=7}\end{table}

\begin{table}\begin{tabular}{|l|l|}
\hline & \\
$\begin{array}{lcl}C_1 &= &1\langle(1,1,1)\rangle + \omega\langle(1,1,1)\rangle^\perp\end{array}$ 
&
$\begin{array}{lcl}C_2 &= & \omega \langle(1,1,\omega+1)\rangle + 1 \langle(1,1,\omega+1)\rangle^\perp\end{array}$
\\ & \\ 
$\begin{array}{lcl}C_1 &=& \{
(0, 0, 0), (\omega, \omega, \omega), (1, 1, 1), \\ & &
(\omega + 1, \omega + 1, \omega + 1), (\omega, \omega, 0), \\ & &(0, 0, \omega), (\omega + 1, \omega + 1, 1), \\ & &(1, 1, \omega + 1), (0, \omega, \omega), (\omega, 0, 0), \\ & &(1, \omega +1, \omega + 1), (\omega + 1, 1, 1), \\ & &(\omega, 0, \omega), (0, \omega, 0), \\ & &(\omega + 1, 1, \omega + 1), (1, \omega + 1, 1)
\}\end{array}$
&
$\begin{array}{lcl}C_2 &=&  \{
(0, 0, 0), (\omega + 1, \omega + 1, 0), (1, 1, 0), \\ & & (\omega, \omega, 0), (0, 0, \omega + 1), \\ & & (\omega + 1, \omega + 1, \omega + 1), (1, 1, \omega + 1), \\ & & (\omega, \omega, \omega + 1), (0, \omega, 1),\\ & &  (\omega + 1, 1, 1), (1, \omega + 1, 1), (\omega, 0, 1), \\ & & (0, \omega, \omega), (\omega + 1, 1, \omega), \\ & & (1, \omega + 1, \omega), (\omega, 0, \omega)
\}\end{array}$
\\ & \\
$\begin{array}{lcl}swe_{C_1} &= & X^3 + Y^3 + 3X^2Z + \\ & & 3Y^2Z  +3XZ^2 + \\ & & 3YZ^2 + 2Z^3\end{array}$
&
$\begin{array}{lcl}swe_{C_2} &=& X^3 + XY^2 + X^2Z + \\ & & 2XYZ + 3Y^2Z + \\ & & 4XZ^2 + 2YZ^2 + 2Z^3\end{array}$
\\ & \\ 
$\begin{array}{ll} \theta_{\Lambda_7(C_1)}(q) = & 1 + 6q^2 + 8q^3 + 48q^4 + \\ & 24q^5 + 88q^6 + 48q^7 + \\ & 138q^8 + 48q^9 + \dots\end{array}$
& 
$\begin{array}{ll} \theta_{\Lambda_7(C_2)}(q) = & 1 + 6q^2 + 8q^3 + 48q^4 + \\ & 24q^5 + 88q^6 + 48q^7 + \\ & 138q^8 + 48q^9 + \dots\end{array}$
\\ & \\
$\begin{array}{ll} \theta_{\Lambda_{15}(C_1)}(q) = & 1 + 8q^3 + 12q^4 + \\ & 30q^6 + 72q^8 + 24q^9 + \\ &  54q^{10} + \dots\end{array}$
&
$\begin{array}{ll} \theta_{\Lambda_{15}(C_2)}(q)  = & 1 + 4q^2 + 8q^4 + 8q^5 + \\ & 34q^6 + 8q^7 + 60q^8 + \\ & 32q^9 + 50q^{10}+ \dots\end{array}$
\\ & \\ \hline
\end{tabular}\caption{Two nonequivalent codes of length $n=3$ over $\mathbb{F}_2\times\mathbb{F}_2$ with the same theta series for $\ell=7$ and different theta series for $\ell>7$.  As given in Example~\ref{ex:new-length3-codes}.}\label{tab:new_example_p=2,n=3,ell=7}\end{table}

\begin{table}\begin{tabular}{|l|l|}
\hline & \\
$\begin{array}{lcl}C_1 &= &1\langle(\omega,\omega,\omega,\omega)\rangle + 1\langle(\omega,\omega,\omega,\omega)\rangle^\perp\end{array}$ 
&
$\begin{array}{lcl}C_2 &= & 1 \langle(\omega,\omega,1,1)\rangle + 1 \langle(\omega,\omega,1,1)\rangle^\perp\end{array}$
\\ & \\ 
$\begin{array}{lcl}swe_{C_1} &= & X^4 + 6X^2Y^2 + Y^4 + \\ & & 12X^2Z^2 + 24XYZ^2 +\\ & & 12Y^2Z^2 +
8Z^4 \end{array}$
&
$\begin{array}{lcl}swe_{C_2} &=& X^4 + 2X^2Y^2 + Y^4 +\\ & &  8X^2YZ + 8XY^2Z + \\ & & 8X^2Z^2 +
8XYZ^2 + 8Y^2Z^2 +\\ & &  8XZ^3 + 8YZ^3 + 4Z^4 \end{array}$
\\ & \\ 
$\begin{array}{ll} \theta_{\Lambda_3(C_1)}(q) = & 1 + 72q^2 + 192q^3 + 504q^4 + \\ &576q^5 + 2280q^6 + 1728q^7 \\ &
+ 4248q^8 + 4800q^9 + \dots\end{array}$
& 
$\begin{array}{ll} \theta_{\Lambda_3(C_2)}(q) = & 1 + 72q^2 + 192q^3 + 504q^4 + \\ &576q^5 + 2280q^6 + 1728q^7 \\ &
+ 4248q^8 + 4800q^9 + \dots\end{array}$
\\ & \\
$\begin{array}{ll} \theta_{\Lambda_{11}(C_1)}(q) = & 1 + 24q^2 + 24q^4 + \\ & 144q^6 + 192q^7 + 312q^8 + \\ & 384q^9 + \dots\end{array}$
&
$\begin{array}{ll} \theta_{\Lambda_{11}(C_2)}(q)  = & 1 + 8q^2 + 56q^4 + \\ & 64q^5 + 96q^6 + 128q^7 + \\ & 344q^8 + 320q^9 + \dots\end{array}$
\\ & \\ \hline
\end{tabular}\caption{Two nonequivalent codes of length $n=4$ over $\mathbb{F}_4$ with the same theta series for $\ell=3$ and different theta series for $\ell>3$.  As given in Example~\ref{ex:p=2,n=4,ell=3}.}\label{tab:p=2,n=4,ell=3}\end{table}

\end{document}